\documentclass{scrartcl}
\usepackage{amsmath}
\usepackage{amssymb}
\usepackage{amsthm}
\usepackage[T1]{fontenc}
\usepackage{enumerate}

\usepackage{hyperref}
\usepackage{breakurl}

\DeclareMathOperator{\diam}{diam}
\DeclareMathOperator{\lip}{Lip}

\newtheorem{proposition}{Proposition}[section]
\newtheorem{corollary}[proposition]{Corollary}

\newtheorem{theorem}[proposition]{Theorem}
\newtheorem{lemma}[proposition]{Lemma}
\theoremstyle{definition}
\newtheorem{definition}[proposition]{Definition}
\theoremstyle{remark}
\newtheorem{remark}[proposition]{Remark}
\newtheorem{example}[proposition]{Example}

\title{On generic convergence of successive approximations of mappings with convex and compact point images}
\author{Christian Bargetz \and Emir Medjic \and Katriin Pirk}

\begin{document}
\maketitle
\begin{abstract}
  \noindent\textbf{\textsf{Abstract.}} We study the generic behavior of the method of successive approximations for set-valued mappings in separable Banach spaces. We consider the case of nonexpansive mappings with convex and compact point images and show that for the typical such mapping and typical points of its domain the sequence of successive approximations is unique and converges to a fixed point of the mapping.
  \vskip2mm
  \noindent\textbf{\textsf{Keywords.}} Banach space, generic property, set-valued nonexpansive mapping, successive approximations.
  \vskip2mm
  \noindent\textbf{\textsf{Mathematics Subject Classifications.}} 47H04, 47H09, 47H10, 54E52
\end{abstract}

\section{Introduction}

Brouwer's fixed point theorem implies that every continuous self-mapping of a bounded, closed and convex subset of a Euclidean space has a fixed point. Unfortunately, this is no longer true in infinite dimensions. There are even nonexpansive self-mappings of bounded closed and convex subsets of a Banach space which do not have a fixed point. This motivates the question of whether this is at least true for typical nonexpansive mappings. In~\cite{BM1976} F.~S.~de~Blasi and J.~Myjak gave a positive answer to this question. More precisely, given a Banach space $X$ and a bounded, closed and convex subset $C\subset X$, they show that the set of nonexpansive self-mappings of $C$ having a unique fixed point is the complement of a meagre set in the space of all nonexpansive self-mappings equipped with the metric of uniform convergence. In addition, they also showed that this fixed point of the nonexpansive mapping $f$ can be reached by starting with an arbitrary point $x_0$ and then iteratively setting $x_{k+1}=f(x_k)$ for $k=1,2,\ldots$.

This motivates the question of whether the typical nonexpansive mapping actually has Lipschitz constant smaller than one. In the case of Hilbert spaces, F.~S.~de~Blasi and J.~Myjak already answered this question negatively in~\cite{BM1976} by showing that to the contrary the typical nonexpansive mapping has the maximal possible Lipschitz constant one. This result was later generalised to Banach spaces in~\cite{BR2016} and more general settings in~\cite{BDR2017}.

In~\cite{RZ2001}, S.~Reich and A.~Zaslavski gave an explanation of de Blasi and Myjak's result by showing that the typical nonexpansive mapping satisfies the assumptions of Rakotch's fixed point theorem~\cite{Rakotch}.

The results on generic existence of fixed points of nonexpansive mappings have been generalised to set-valued mappings in~\cite{BMRZ2009} and~\cite{PL2015}. Recall that a fixed point of a set-valued mapping
\[
  F\colon C\to 2^{C}\setminus\{\emptyset\}
\]
is an element $x\in C$ satisfying $x\in F(x)$. The proofs of these results use iterations of a mapping defined on suitable hyperspaces, i.e. spaces of sets equipped with the Hausdorff distance. More precisely, let $\mathcal{B}(C)$ be a hyperspace of certain nonempty and closed subsets of $C$, and let $F\colon C\to\mathcal{B}(C)$ be a nonexpansive mapping. Then the mapping
\[
  \tilde{F}\colon \mathcal{B}(C)\to\mathcal{B}(C), \qquad A \mapsto \bigcup_{x\in A} F(x),
\]
is considered. In some sense this means that the problem is transferred to a question on a self-mapping defined on a metric space of sets.

In~\cite{Pianigiani} the problem is approached from a different angle. In the case of rather simple set-valued mappings which are defined by pairs of nonexpansive mappings the following iterative procedure is considered:
\[
  x_0\in C, \qquad x_{k+1} \in \operatorname{argmin} \{\|y-x_k\|\colon y\in F(x_k)\}, \quad k\in\mathbb{N}.
\]
Under the assumption that $C$ is a bounded, closed and convex subset of a Hilbert space, G.~Pianigiani showed that, given a point $x_0$ the set of nonexpansive mappings for which the above sequence is uniquely defined and converges to a fixed point is the complement of a meagre set. In~\cite{BR2019}, using different methods, these results have been generalised to the setting of Banach spaces. 

More generally, in~\cite{Medjic2022} compact-valued nonexpansive mappings on bounded, closed and convex subsets of Banach spaces are considered. The main results of this article establish that the typical compact-valued nonexpansive mapping has a unique sequence of successive approximations for typical initial points in its domain. Moreover this sequence converges to a fixed point of this mapping.

The current article is considered with the same question but for mappings whose point-images are compact and convex. A crucial step in the proof in~\cite{Medjic2022} is to perturb a given nonexpansive mapping $F$ by adding an additional point to $F(x)$ for some $x$. This point $z$ is required to have positive distance to $F(x)$ rendering $F(x)\cup\{z\}$ non-convex. For this reason a  different approach is needed for the case of compact and convex sets. 

In order to see that this is not a particular case of the situation considered in~\cite{Medjic2022} we show in Proposition~\ref{prop:ConvPorous} that the set of convex compact sets is a porous and hence small subset of the hyperspace of compact sets.

One of the main methods of the current construction is based on V. Klee's proof that the typical convex body is smooth and rotund, see~\cite{Klee1959SmoothnessConvexity}. A strengthening of this result for convex bodies in Euclidean spaces was later obtained by T. Zamfirescu in~\cite{Zamfirescu}. A survey on typical properties of convex bodies can be found in~\cite{Gruber}.

\section{Preliminaries and notation}

In this section we clarify most of the main notions and notation used throughout this paper. Some definitions, essential to other sections, are introduced where they are needed. We consider mostly Banach spaces and hyperspaces constructed of their subsets. Given a Banach space $X$, we denote the open ball with centre $a\in X$ and radius $r>0$ by $B_X(a,r),$ and the corresponding closed ball by $\overline{B}_X(a,r).$ If there is no confusion, we omit the subscript and write e.g. $B(a,r)$.
Given $x\in X$ and $A\subset X$ we write $d(x,A)=\inf_{y\in A}\{\|x-y\|\}$ and $\diam A=\sup_{x,y\in A}\{\|x-y\|\}$ for the distance of the point $x$ from the set $A$ and the diameter of the set $A$, respectively.

Next up we construct a suitable hyperspace for our purposes. Let $C\subset X$ be a nonempty closed and bounded set. Consider the set $\mathcal{CK}(C)$ of all nonempty convex compact subsets of $C$.  It is well-known that $(\mathcal{CK}(C),h)$ is a complete metric space, where $h$ stands for the Hausdorff distance defined in the following way:
\[
h(A,B):=\max\{\sup_{a\in A}\inf_{b\in B}\|x-y\|,\sup_{b\in B}\inf_{a\in A}\|x-y\|\}, \quad \text{for all}\quad A,B\in \mathcal{CK}(C).
\]
Note that there is an analogue to the triangle inequality that combines the distance to a point and the Hausdorff distance, i.e.:
\[
d(x,A)\leq d(x,B)+h(B,A)\quad \text{for all}\quad A,B\in \mathcal{CK}(C),x \in X.
\]

Note that in an analogous way one can define the hyperspace $\mathcal{K}(C)$ of nonempty compact subsets of $C$ with respect to Hausdorff metric. In this paper we are concerned with typical properties, i.e. properties for which the set of elements enjoying it is the complement of a small set. The size of a set can be regarded in many ways, e.g. in the sense of Baire category, \emph{meagre} sets, i.e. countable unions of nowhere dense sets, are considered to be small. We also use the notion of porous sets that are in the same scale even smaller.
More precisely, for an arbitrary metric space $M$ we say that a subset $A\subset M$ is \emph{porous at point} $x\in A$ if there are constants $\alpha>0$ and $\varepsilon_0>0$ that satisfy the following condition: for all $\varepsilon\in(0,\varepsilon_0)$ there exists $y\in B(x,\varepsilon)$ such that $B(y,\alpha \varepsilon)\cap A=\emptyset$. We say the set $A$ is \emph{porous} provided it is porous at all of its points. $A$ is called \emph{$\sigma$-porous} if it is a countable union of porous sets.

Recall that in \cite{Medjic2022} similar results to the main results of the current paper were obtained. There the author handles the case of compact-valued nonexpansive mappings, whereas in this paper we regard nonexpansive mappings with values that are compact and convex. Despite the seeming overlapping, the small size of the set of convex compact subsets implies that the formally narrower approach is meaningful on its own.

\begin{proposition}\label{prop:ConvPorous}
  Let $X$ be a Banach space and $C\subset X$ a bounded closed and convex subset. The set of convex compact subsets of~$C$ is a porous subset of the hyperspace $\mathcal{K}(C)$ of compact subsets of~$C$.
\end{proposition}

\begin{proof}
  Fix an arbitrary convex and compact $K\in\mathcal{K}(C)$~. If $K\neq C$ we set
  \[
    \alpha=\frac14 \qquad\text{and}\qquad \varepsilon_0 := \sup\{d(x,K)\colon x\in C\}.
  \]
Observe that $\varepsilon_0$ is positive. Given $\varepsilon\in(0,\varepsilon_0)$ we now pick $z\in C$ with $d(z,K)=\frac34\varepsilon$ and set $K' = K\cup\{z\}$ which is a compact set with $h(K,K')<\varepsilon$. Now observe that no compact set $L$ with $h(K',L)<\alpha\varepsilon$ can be convex. In the case of $K=C$, which of course can only happen if $C$ is compact itself, we only change the value of $\varepsilon_0$ to $\varepsilon_0=\diam C/3$~. Then for $\varepsilon\in(0,\varepsilon_0)$ we set $K'=(K\setminus B(z,\varepsilon))\cup\{z\}$ for some arbitrary $z\in C$ and confirm that again no compact set $L$ with $h(K',L)<\alpha\varepsilon$ can be convex. Since $K$ was chosen arbitrarily, the result follows.
\end{proof}

As mentioned, the core results of this paper are for certain set-valued nonexpansive mappings. We call a mapping $F\colon C\to \mathcal{CK}(C)$ \emph{nonexpansive} if for all $x,y\in C$
\[
h(F(x),F(y))\leq \|x-y\|. 
\]
We denote the set of all such nonexpansive mappings by $\mathcal{M}$. It is well-known that when equipped with the metric of uniform convergence, i.e. for all $F,G\in \mathcal{M}$
\[
d_\infty(F,G):= \sup_{x\in C} h(F(x),G(x)),
\]
$(\mathcal{M},d_\infty)$ is a complete metric space. Similarly to the vector-valued case, we may regard Lipschitz mappings among set-valued mappings. Note that if $F\in \mathcal{M}$, i.e. $F$ is a nonexpansive mapping, then its Lipschitz constant satisfies $\lip F\leq 1$, in the case of $\lip F<1$ we say that $F\in \mathcal{M}$ is a \emph{strict contraction}. In Section 3 (see Proposition \ref{prop: NdenseinM}) we see that all such strict contractions form a dense subset of $\mathcal{M}$.

The following fact regarding Lipschitz mappings on geodesic spaces is well known, nonetheless we give it with proof for the convenience of the reader.
\begin{lemma}\label{lem: Lipboundinandoutofball}
Let $(Y,\rho)$ be a geodesic space and $C\subset Y$ a convex subset. For a continuous function $f\colon Y\to Y$ to be Lipschitz it is necessary and sufficient that the restrictions $f|_C$ and $f|_{Y\backslash C}$ are both Lipschitz. Moreover, the Lipschitz constant of $f$ has the following bound
\[
\lip f\leq \max\{\lip (f|_C),\lip (f|_{Y\backslash C})\}.
\]
\end{lemma}
\begin{proof}
Necessity is obvious. To show the sufficiency and the claimed bound to the \mbox{Lipschitz} constant it is enough to show the result for fixed $x\in C$ and $y\in Y\backslash C$, since the other cases are trivial. Now, we may fix a metric segment $[x,y]$ and due to the fact that $C$ is convex, there exists a unique point $z$ in the intersection of $[x,y]$ with the closures of $C$ and $Y\backslash C$, therefore $\rho(x,y)=\rho(x,z)+\rho(z,y)$ and furthermore
\begin{align*}
\rho(f(x),f(y)) &\leq \rho(f(x),f(z))+\rho(f(z),f(y))\leq\lip (f|_C)\rho(x,z)+\lip(f|_{Y\backslash C})\rho(z,y)\\&
\leq \max\{\lip (f|_C),\lip (f|_{Y\backslash C}\rho(x,y),
\end{align*}
which completes the proof.

\end{proof}
\section{Hyperspaces of convex sets}
The main aim of this section is to establish the following theorem on the hyperspace of bounded closed and convex subset of a given closed and convex subset of a Banach space.
\begin{theorem}\label{thm: CBCishyperbolic}
  Let $X$ be a Banach space and let $D\subset X$ be a closed convex set. The hyperspace $\mathcal{CB}(D)$ of bounded closed and convex subsets of~$D$ equipped with the Hausdorff distance is a complete hyperbolic metric space.
\end{theorem}

The following immediate consequence of this theorem states the same for the hyperspace $\mathcal{CK}(D)$. 
\begin{corollary}
Let $X$ be a Banach space and let $D\subset X$ be a closed convex set. The hyperspace $\mathcal{CK}(D)$ of convex compact subsets of~$D$ equipped with the Hausdorff distance is a complete hyperbolic metric space.
\end{corollary}
We start by introducing a few notions regarding geodesic metric spaces. We say that a metric space $(Y,\rho)$ is \emph{geodesic} if for every $x,y\in Y$ there exists an isometric embedding $c\colon [0,\rho(x,y)]\to Y$ such that $c(0)=x$ and $c(\rho(x,y))=y$. The image of such embedding is called a \emph{metric segment} in $Y$ with endpoints $x$ and $y$. Metric segments between two points need not be unique. If there is a unique metric segment with endpoints $x$ and $y$, then we denote it by $[x,y].$

For example in a closed convex subset $C_{\ell_1}\subset \ell_1$ that has at least three non-colinear points, it is easy to imagine different metric segments between a pair of points. However, thanks to the underlying vector space structure, we may select a family of metric segments by choosing for each pair $x,y\in C_{\ell_1}$ as the unique segment the metric segment of the form
\[
[x,y]:= \{\lambda x+(1-\lambda)y\colon \lambda\in [0,1]\}.
\]
In the following we assume that we have a geodesic metric space together with a family~$\mathcal{S}$ of metric segments containing a unique one for each pair of points.
Given a unique metric segment $[x,y]$ in $(Y,\rho)$ and $\lambda\in [0,1]$ we denote by $(1-\lambda)x\oplus \lambda y$ the unique point $z\in [x,y]$ which satisfies both $\rho(z,x)=\lambda \rho(x,y)$ and $\rho(z,y)=(1-\lambda)\rho(x,y)$.
\begin{definition}\label{def:hyperbolicspace}
Let $(Y,\rho)$ be a metric space and $\mathcal{S}$ a family of metric segments in $Y$. We call the triple $(Y,\rho,\mathcal{S})$ \emph{hyperbolic} if the following conditions are satisfied:
\begin{itemize}
\item[$(i)$] For each pair $x,y\in Y$, there exists a unique metric segment $[x,y]\in \mathcal{S}$ joining $x$ and $y$.
\item[$(ii)$] For all $x,y,z,w\in Y$ and all $t\in [0,1]$,
\[\rho((1-t)x\oplus t y,(1-t)w\oplus tz)\leq (1-t)\rho(x,w)+t\rho(y,z).\]
\item[$(iii)$] The collection $\mathcal{S}$ is closed with respect to subsegments. More precisely, for all $x,y\in X$ and $u,v\in [x,y]$ we have $[u,v]\subset [x,y].$
\end{itemize}
\end{definition}

In the literature there is a number of different definitions of hyperbolic spaces. In this paper we use the same approach as the authors of \cite{BDR2017} (for more information on different notions of hyperbolicity see e.g. Remark 2.13 in \cite{Kohlenbach}).

We can almost start gathering the tools to prove the main result of this section, but firstly, we would like to point out that Example 4.7 in \cite{Strobin2014} and also in Remark 5.2 in \cite{BDR2017} show that in the statement of Theorem \ref{thm: CBCishyperbolic} the assumption that the sets are convex is necessary. The first example mentioned above shows that the hyperspace of bounded and closed subsets even in a Banach space cannot be hyperbolic. For the convenience of the reader we repeat the second example here again.

\begin{example}
  Consider the metric space $D=[-1,1]^2$ with $\|\cdot\|$ the Euclidean norm. \linebreak Set $A=\{(-1,-1), (-1,1)\}$ and $B=\{(1,-1),(1,1)\}$. Then $h(A,B) = 2$ and 
  \[
    \frac{1}{2}A + \frac{1}{2}B = \{(0,-1), (0,0), (0,1)\}.
  \]
  Therefore 
  \[
    h\left(\frac{1}{2}A + \frac{1}{2}B,A\right) = \sqrt{2} \not = \frac{1}{2}h(A,B)
  \]
  which shows that taking convex combinations of sets does not result in geodesics in the hyperspace of compact subsets of~$D$.
\end{example}
We consider a Banach space $X$ and bounded, closed and convex sets $A,B \subset X$. Note that it is enough to consider this case since the following constructions preserve being subset of a given closed and convex subset. Given $\lambda \in [0,1]$ we set
\[
  \lambda A + (1-\lambda) B := \overline{\{\lambda x + (1-\lambda)y \colon x\in A, y\in B\}}
\]
and observe that this set is obviously bounded, convex and closed. Also note that if $X$ is reflexive or one of the sets is compact, we do not need to take the closure in the definition of the above set.
Next we present three lemmas that basically build the proof for Theorem \ref{thm: CBCishyperbolic}, the latter proof can be found right after these lemmas. Although some of the following lemmas are well-known we include their proofs for the convenience of the reader.

In the proofs of the following lemmas we use that the Hausdorff distance between two sets is the same as the Hausdorff distance between their closures.

\begin{lemma}\label{lem:Enpoints}
  The set $\lambda A + (1-\lambda) B$ satisfies 
  \begin{equation*}
    h(A, \lambda A + (1-\lambda) B) = (1-\lambda) h(A,B) \quad \text{and} \quad
    h(B, \lambda A + (1-\lambda) B) = \lambda h(A,B)
  \end{equation*}
  for every $\lambda\in[0,1]$.
\end{lemma}

\begin{proof}
  We prove the first equality, the second can be shown similarly. For every $x\in A$ we have
  \[
    d(x,\lambda A + (1-\lambda) B) \leq \|x-(\lambda x + (1-\lambda) y)\| = (1-\lambda) \|x-y\|
  \]
  for all $y\in B$ and hence $d(x,\lambda A + (1-\lambda) B) \leq (1-\lambda) d(x,A) \leq (1-\lambda) h(A,B)$. On the other hand, for $\lambda x+ (1-\lambda)y \in \lambda A + (1-\lambda)B$ we have
  \[
    d(\lambda x+ (1-\lambda)y, A) \leq \lambda d(x,A) + (1-\lambda) d(y,A) = (1-\lambda) d(y,A) \leq (1-\lambda) h(A,B).
  \]
  Summing up, this shows that
  \[
    h(A, \lambda A + (1-\lambda) B ) \leq (1-\lambda) h(A,B).
  \]
  The converse inequality is obtained by triangle inequality. Indeed,
  \[
    h(A, \lambda A + (1-\lambda) B) \geq h(A,B) - h(B,\lambda A + (1-\lambda) B)) \geq (1-\lambda) h(A,B).
  \]
\end{proof}

\begin{lemma}\label{lem:CBGeodesic}
  For $\lambda,\mu \in[0,1]$ we have
  \[
    \lambda A + (1-\lambda) (\mu A + (1-\mu)B) = (\lambda + (1-\lambda)\mu) A + (1-\lambda)(1-\mu) B
  \]
  and 
  \[
    h( \lambda A + (1-\lambda) B, \mu A + (1-\mu) B) = |\lambda-\mu| h(A,B).
  \]
\end{lemma}

\begin{proof}
  First note that for $x\in A$ and $y\in B$ we have
  \[
    (\lambda + (1-\lambda)\mu) x+ (1-\lambda)(1-\mu) y = \lambda x + (1-\lambda)(\mu x + (1-\mu)y) \in \lambda A + (1-\lambda) (\mu A + (1-\mu) B)
  \]
  and hence $(\lambda + (1-\lambda)\mu) A + (1-\lambda)(1-\mu) B \subset \lambda A + (1-\lambda) (\mu A + (1-\mu)B)$. In order to see the converse inclusion let $x_1,x_2\in A$ and $y\in B$. We have
  \begin{multline*}
    \lambda x_1 + (1-\lambda)(\mu x_2 + (1-\mu)y)\\  = (\lambda+(1-\lambda)\mu) \left(\frac{\lambda}{\lambda+(1-\lambda)\mu} x_1 + \frac{(1-\lambda)\mu}{\lambda+(1-\lambda)\mu} x_2\right) + (1-\lambda)(1-\mu)y\\
    \in (\lambda + (1-\lambda)\mu) A + (1-\lambda)(1-\mu) B
  \end{multline*}
  since $A$ is convex.
  
  For the proof of the second equality, we may assume without loss of generality that $\lambda \geq \mu$ and observe that
  \[
    \lambda A + (1-\lambda) B = sA + (1-s)(\mu A + (1-\mu)B)\qquad\text{for}\qquad s = \frac{\lambda-\mu}{1-\mu}.
  \]
  and hence
  \begin{align*}
    h(\lambda A+(1-\lambda)B, \mu A + (1-\mu) B) &= h(sA + (1-s)(\mu A + (1-\mu) B), \mu A + (1-\mu) B)\\
                                                 &= s h(A, \mu A + (1-\mu) B) = \frac{\lambda-\mu}{1-\mu}(1-\mu) h(A,B)\\
                                                 &= |\lambda-\mu| h(A,B)
  \end{align*}
  by Lemma~\ref{lem:Enpoints}.
\end{proof}

\begin{lemma}\label{lem:BCHyperbolic}
  Given bounded closed and convex sets $A,B,C$ and $\lambda \in[0,1]$ the inequality
  \[
    h(\lambda A+(1-\lambda) B, \lambda A + (1-\lambda) C) \leq (1-\lambda) h(B,C)
  \]
  holds.
\end{lemma}

\begin{proof}
  For $\lambda x + (1-\lambda)y \in \lambda A + (1-\lambda) B$ we have
  \begin{align*}
    d(\lambda x + (1-\lambda)y, \lambda A + (1-\lambda) C) &\leq \|\lambda x + (1-\lambda) y - (\lambda x + (1-\lambda)z )\| \\ &\leq (1-\lambda) \|y-z\|
  \end{align*}
  for all $z\in C$. Therefore, $d(\lambda x + (1-\lambda)y, \lambda A + (1-\lambda) C) \leq  (1-\lambda) d(y,C)$. Similarly, we may conclude that $d(\lambda x + (1-\lambda)z, \lambda A + (1-\lambda) B) \leq  (1-\lambda) d(z,B)$. Combining these inequalities, we obtain
  \[
    h(\lambda A+(1-\lambda) B, \lambda A + (1-\lambda) C) \leq (1-\lambda) h(B,C)
  \]
  as claimed.
\end{proof}

\begin{proof}[Proof of Theorem \ref{thm: CBCishyperbolic}]
We consider the hyperspace $\mathcal{CB}(D)$ of bounded and convex subsets of $D$ with the Hausdorff distance. We pick for each pair $A,B\subset D$ the unique metric segment of the form
\[
[A,B]=\{\lambda A + (1-\lambda) B\colon \lambda\in [0,1]\}
\]
and consider $\mathcal{S}$ to be the collection of such metric segments. The fact that each $[A,B]$ is a metric segment, follows from Lemmas \ref{lem:Enpoints} and \ref{lem:CBGeodesic}. Due to the construction, these metric segments also contain all of their subsegments. To finish the proof, note that the condition $(ii)$ of Definition \ref{def:hyperbolicspace} follows from Lemma \ref{lem:BCHyperbolic}.
\end{proof}

Another useful consequence that can be derived from Lemma \ref{lem:BCHyperbolic} is that in our setting set-valued strict contractions are dense in the set of set-valued nonexpansive mappings. 

\begin{proposition}\label{prop: NdenseinM}
The set of strict contractions of the form $F\colon C\to \mathcal{CK}(C)$ is dense in the set of all nonexpansive mappings $G\colon C\to \mathcal{CK}(C)$.
\end{proposition}
\begin{proof}
Fix an arbitrary nonexpansive mapping $G\colon C\to \mathcal{CK}(C)$, $\delta>0$, and $\gamma>0$ such that $\gamma\diam(C)<\delta$. Choose some bounded closed and convex $A\subset C$ and define for $G$ a mapping $F\colon C\to \mathcal{CK}(C)$ by setting $F(x):=(1-\gamma)G(x)+\gamma A$ for every $x\in C$. Note that by Lemma~\ref{lem:BCHyperbolic} $F$ is a strict contraction and by Lemma~\ref{lem:Enpoints} we get for every $x\in C$ that
$$h(G(x),F(x))=h(G(x),(1-\gamma)G(x)+\gamma A)=\gamma h(G(x),A)\leq \gamma\diam C<\delta,$$
because for any $A,B\subset C$ we have $h(A,B)\leq \diam (C)$. This concludes the result.
\end{proof}

We end this section with the following result that later on allows us to perturb a set-valued mapping without losing control over its Lipschitz constant.

\begin{proposition}\label{prop:ConvexLip}
  Let $X$ be a Banach space, $C\subset X$ a closed convex subset, $A\subset C$ a bounded, closed and convex set and $F\colon C\to \mathcal{CK}(C)$ a nonexpansive mapping. For a Lipschitz mapping $\lambda\colon C\to [0,1]$ the inequality
  \begin{multline*}
    h(\lambda(x)A+(1-\lambda(x)) F(x), \lambda(y)A+(1-\lambda(y)) F(y)) \\\leq (\lip F + h(A,F(y))\lip \lambda) \|x-y\|
  \end{multline*}
  is satisfied.
\end{proposition}

\begin{proof}
  Using the triangle inequality together with Lemmas~\ref{lem:CBGeodesic} and~\ref{lem:BCHyperbolic} we obtain
  \begin{align*}
    h(\lambda(x) A + (1-\lambda(x)) &F(x), \lambda(y) A + (1-\lambda(y)) F(y))\\
                                          &\leq h(\lambda(x) A+ (1-\lambda(x)) F(x), \lambda(x) A + (1-\lambda(x))F(y)) \\
                                          & \qquad + h(\lambda(x) A + (1-\lambda(x))F(y), \lambda(y) A + (1-\lambda(y)) F(y))\\
                                    & \leq (1-\lambda(x)) h(F(x), F(y)) + |\lambda(x)-\lambda(y)| h(A,F(y)) \\
                                    &\leq (\lip F + h(A,F(y)) \lip \lambda) \|x-y\|
  \end{align*}
  which is the required inequality.
\end{proof}

\begin{remark}
We have stated Propositions \ref{prop: NdenseinM} and \ref{prop:ConvexLip} according to the needs of this paper. Note that the proofs of these statements actually do not use compactness, therefore these results also hold for the more general case of the hyperspace $\mathcal{CB}(C)$ instead of $\mathcal{CK}(C)$. 
\end{remark}

\section{Continuity Properties of Metric Projections}

We are interested in convex sets with the property that each point has a unique projection onto this set. For this reason we use the following notion of rotundity which is a slight modification of $X^*$-rotundity introduced by V.~Klee in~\cite{Klee1959SmoothnessConvexity}.

Given a Banach space $X$, a nonzero functional $x^*\in X^*$ , and $r\in \mathbb{R}$ we say a hyperplane $H=\{x\in X\colon \langle x^*,x\rangle=r\}$ is supporting $C\subset X$ at a point $z\in C$ if $\langle x^*,z\rangle =r$ and either $\sup_{x\in C}\langle x^*,x\rangle\leq r$ or  $\inf_{x\in C}\langle x^*,x\rangle\geq r$. Note that in Klee's work, it is demanded here also that the set $C$ is not contained in $H$, this is where our concept diverges a little from that of Klee's. Hence, our definition of rotundity is more restrictive since in contrast to Klee  we formally have more supporting hyperplanes for any point.

\begin{definition}
  Let $X$ be a Banach space and $C\subset X$ be a convex set. We call $C$ \emph{rotund} if for every $z\in C$ and every hyperplane $H$ supporting $C$ in $z$ we have that $H\cap C = \{z\}$. A point $z\in C$ is called \emph{support point} if there is a hyperplane supporting $C$ in $z$, and a \emph{non-support point} otherwise.
\end{definition}

The following characterisation of support points will come in handy and shows that points that reduce the distance to a point outside a convex set in a certain geometrical sense behave like the topological boundary of this set. This characterisation is a direct consequence of the characterisation of the (set-valued) metric projection via a variational inequality. Since the following direct proof is rather short, we include it nevertheless for the convenience of the reader.

\begin{proposition}\label{prop:HyperplaneAndProjection}
  Let $X$ be a Banach space and $C\subset X$ a closed convex set. A point $z\in C$ is a support point if and only if there is a point $x\not\in C$ with $\|x-z\|=d(x,C)$.
\end{proposition}

\begin{proof}
  Assume that there is $x\in X\setminus C$ with $\|x-z\|=d(x,C)$. Then the open ball $B(x,\|x-z\|)$ and $C$ are disjoint convex sets which by the Hahn-Banach Theorem can be separated by a closed hyperplane which necessarily has to contain the point $z$. This hyperplane is the required supporting hyperplane.

  If on the other hand there is a functional $x^*\in X^*$ and $a\in \mathbb{R}$ with $\langle x^*,z \rangle = a \geq \langle x^*,y\rangle$ for all $y\in C$, we denote by $H$ the hyperplane defined by $x^*$ and $a$ and pick $\tilde{x}$ with $\langle x^*, \tilde{x}\rangle <a$ and a point $\tilde{z}$ with $\|\tilde{x}-\tilde{z}\|=d(\tilde{x},H)$. Setting $x = \tilde{x} + z-\tilde{z}$ we observe that
  \[
    \|x-z\| = \|\tilde{x}-\tilde{z}\| = d(\tilde{x},H) = d(x,H).
  \]
  Since $C$ and $x$ are on different sides of $H$, we conclude that $\|x-z\|=d(x,C)$.
\end{proof}

The following proposition shows that rotundity is indeed the property we need to additionally consider in order to guarantee the uniqueness of projections, since then for whatever point we choose in the space, there is a unique element in that rotund set that realises its distance to the same set.
\begin{proposition}\label{prop:StrConProjUnique}
  Let $X$ be a Banach space and $C\subset X$ be a rotund weakly compact set. For every $x\in X$ there is a unique $z\in C$ with $\|x-z\|=d(x,C)$, i.e. a unique point minimising the distance to~$C$.
\end{proposition}

\begin{proof}
  First note that since $C$ is weakly compact, the set of points in $C$ minimising the distance to $x$ is nonempty. Assume we have $z_1,z_2\in C$ with
  \[
    \|x-z_1\| = d(x,C) = \|x-z_2\|.
  \]
  Since the norm is a convex function we may conclude that also $\frac{1}{2}(z_1+z_2)$ minimises the distance to $x$. We pick a supporting hyperplane $H$ in $\frac{1}{2}(z_1+z_2)$ defined by $x^*\in X^*$ and $a\in \mathbb{R}$, i.e. $H=\{x\in X\colon \langle x^*,x\rangle = a\}$. We observe that
  \[
    a \leq \langle x^*, z_1\rangle = 2 \left(\langle x^*, \frac{1}{2}(z_1+z_2)\rangle - \frac{1}{2} \langle x^*, z_2\rangle \right) \leq a
  \]
  which by rotundity implies that $z_1=z_2$.
\end{proof}

\begin{definition}
  Let $X$ be a Banach space, $C\subset X$ a rotund weakly compact set and $x\in X$. The unique element of $C$ minimising the distance to $X$ is denoted by $P_Cx$. The mapping
  \[
    X \to C, \qquad x \mapsto P_Cx
  \]
  is called the \emph{metric projection onto $C$}.
\end{definition}

The following result regarding the continuity of the metric projection is a particular case of Theorem~3 in~\cite{Wulbert1968} but since the proof is rather short, we include it for the convenience of the reader.

\begin{proposition}
  Let $X$ be a Banach space and $C\subset X$ a rotund compact set. The metric projection
  \[
    P_C\colon X \to C, \qquad x\mapsto P_Cx
  \]
  is continuous.
\end{proposition}

The proof of is basically the one given for Proposition~3.6 in~\cite{BMT2021:MetriProjections} where the results is stated for rotund compact sets with nonempty interior.

\begin{proof}
  Let $x\in X$ and $(x_k)_{k\in\mathbb{N}}$ be a sequence converging to $x$. We denote $z_k:=P_{C}x_k$. Since $(z_k)_{k\in\mathbb{N}}$ sits in the compact set $C$ it  has a convergent subsequence $(z_{k_m})_{m\in\mathbb{N}},$ i.e. there is $z\in C$ such that $z_{k_m}\rightarrow z.$ Observe that
  \[
    \|x-z\| = \lim_{m\to\infty}\|x_{k_m}-z_{k_m}\|=\lim_{m\to\infty}d(x_{k_m},C)=d(x,C).
  \]
  Consequently, since $C$ is rotund, $z=P_Cx$ by Proposition~\ref{prop:StrConProjUnique} and $P_Cx$ is the only accumulation point of the sequence $(z_k)_{k\in\mathbb{N}}$. Since it is contained in a compact set, this means that $z_k\to P_Cx$ as claimed (see e.g. the Corollary to Theorem~1 in~\cite[p.~85]{Bou1998:Topology1-4}).
\end{proof}

In order to obtain continuity also with respect to compact sets, we need the following technical lemma.

\begin{lemma}\label{lem:CompSeqUnion}
  Let $(C_k)_{k\in\mathbb{N}}$ be a sequence of compact sets converging to a compact set~$C$ in Hausdorff distance. The set
  \[
    C \cup \bigcup_{k=1}^{\infty} C_k
  \]
  is compact.
\end{lemma}

\begin{proof}
  We show that this set is complete and totally bounded. Given $\varepsilon> 0$ we choose $N\in\mathbb{N}$ with $h(C_n,C)<\varepsilon/2$ for all $n\geq N$. Since $C$ is compact, we may find finitely many points $z_1,\ldots,z_K\in C$ with $C \subset \bigcup_{k=1}^{K} B(z_k,\varepsilon/2)$. Now, we have
  \[
    D:=C \cup \bigcup_{n= N}^{\infty} C_n \subset \bigcup_{k=1}^{K} B(z_k,\varepsilon)
  \]
  since for every $y\in C_n$ if $n\geq N$ we may pick a $z\in C$ with $\|y-z\|<\varepsilon/2$ and a $z_k$ with $\|z-z_k\|<\varepsilon/2$. As a finite union of compact sets $\bigcup_{n=1}^{N-1} C_{n}$ is compact and hence we can add finitely many balls of radius $\varepsilon$ to extend the above to a finite cover of the whole set. In order to show that the set is closed, let $(x_k)_{k\in\mathbb{N}}$ be a sequence in~$D$ converging to some $x\in X$. If the sequence remains in finitely many of the $C_n$, it is clear that also $x$ has to belong to one of these. In the other case, we may pick a subsequence $(x_{k_m})_{m\in\mathbb{N}}$ with $k_{m}>k_{m-1}$ and $x_{k_m}\in C_{k_m}$. Since
  \[
    d(x_{k_m},C) \leq h(C_{k_m},C) \to 0
  \]
  we conclude that $d(x,C)=0$ and hence $x\in C\subset D$, as required.
\end{proof}

\begin{proposition}\label{prop:contMetrProjSaP}
  Let $(C_k)_{k\in\mathbb{N}}$ be a sequence of compact rotund sets converging to a compact rotund set $C$. Then,
  \[
    P_{C_k}x_k\to P_Cx
  \]
  for every $x_k\to x$ in $X$. More generally, if the $C_k$ are convex compact sets (not necessarily rotund) and $z_k\in C_k$ with $\|x_k-z_k\|=d(x_k,C_k)$, we have $z_k\to P_Cx$.
\end{proposition}

\begin{proof}
  We use the notation
  \[
    z_k := P_{C_k}x_k
  \]
  and note that the sequence $(z_k)_{k\in\mathbb{N}}$ is contained in $C \cup \bigcup_{k=1}^{\infty} C_k$ and the latter set is compact by Lemma~\ref{lem:CompSeqUnion}. Hence, there is a convergent subsequence $z_{k_m}\to z$. Note that by definition $z_{k_m}\in C_{k_m}$ and therefore
  \[
    d(z,C) = \lim_{m\to\infty} d(z_{k_m},C) \leq \lim_{m\to\infty} h(C_{k_m},C) = 0
  \]
  and consequently $z\in C$. Moreover, 
  \begin{align*}
    \|x-z\| &= \lim_{m\to\infty} \|x_{k_m}-z_{k_m}\| = \lim_{m\to\infty} d(x_{k_m},C_{k_m}) \leq  \lim_{m\to\infty} (d(x_{k_m},C) + h(C,C_{k_m}))\\
    & = d(x,C)
  \end{align*}
  where the last inequality follows from $d(x,C)\leq d(x,C_{k_m}) + h(C_{k_m},C)$ and the corresponding inequality where the roles of $C$ and $C_{k_m}$ are exchanged. Since $C$ is rotund, we may now use Proposition~\ref{prop:StrConProjUnique} to conclude that $z=P_Cx$. We have shown that $P_Cx$ is the only accumulation point of the sequence $z_k$ and hence the sequence, as it is contained in a compact set, converges to $P_Cx$ (see e.g. the Corollary to Theorem~1 in~\cite[p.~85]{Bou1998:Topology1-4}).
\end{proof}

\begin{remark}
  Given a sequence $(C_k)_{k\in\mathbb{N}}$ of convex compact sets converging to a rotund compact set~$C$ and sequence $x_k\to X$, using the notation
  \[
    \tilde{P}_{C_k}x_k = \{z\in C_k\colon \|x_k-z\| = d(x_k,C_k)\}, 
  \]
  the proof of the above proposition, shows that $h(\tilde{P}_{C_k}x_k, P_Cx) \to 0$ for $k\to\infty$.
\end{remark}

The rest of the section is dedicated to showing that the typical compact convex subset of a separable $C$ in the hyperspace $\mathcal{CK}(C)$ is actually rotund. In order to do so, we follow in general the ideas of V.~Klee from \cite{Klee1959SmoothnessConvexity}, where he showed that a similar result holds for the case of $X^*$-rotundity. For clarity, we state many of the results with proofs, modifying them according to our purposes.

\begin{proposition}[Part of 1.5~Theorem in~\cite{Klee1959SmoothnessConvexity}]\label{prop:KleeStrConv}
  Let $X$ be a separable Banach space and $C\subset X$ a bounded closed and convex set with at least one non-support point. For every $\varepsilon \in (0,1)$ there is a bounded rotund and convex set $K$ with $(1-\varepsilon) C \subset K \subset C$. 
\end{proposition}

The following proof is basically a more detailed version of the one of Lemma~1.4 in~\cite{Klee1959SmoothnessConvexity} and (1.2) of IV in~\cite{Klee1953Homeomorphisms}. Since this part is not proved in~\cite{Klee1959SmoothnessConvexity} we include it for the convenience of the reader.

\begin{proof}
  Without loss of generality, we assume that $0\in C$ and that it is a non-support point. Since $X$ is separable, by Theorem~1.6 in~\cite{Klee1959SmoothnessConvexity} there is an equivalent strictly convex norm $\|\cdot\|_{sc}$ on $X$. We set
  \[
    a := \frac{\varepsilon}{2(1-\varepsilon)\sup\{\|z\|_{sc}\colon z\in C\}}
  \]
  and define the mappings
  \[
    g\colon [0,\infty) \to \mathbb{R}, \qquad t \mapsto \frac{1}{1+at} \qquad\text{and}\qquad T\colon X\to X, \qquad x\mapsto g(\|x\|_{sc}) x.
  \]
  Note that $g(t)\leq 1$ and
  \begin{align*}
    g(tr+(1-t)s) &= \frac{1}{1+a(tr+(1-t)s)} = \frac{1}{t(1+ar)+(1-t)(1+as)} \\
                 &= \frac{g(r)g(s)}{tg(s)+(1-t)g(r)}
  \end{align*}
  for $r,s\geq 0$ and $t\in [0,1]$. We set $K:=T(C)$ and observe that $T(C)\subset C$ since $Tz$ is a convex combination of $0$ and $z$. Note that for $u \in C$ and $\lambda \in [0,1]$ we have $T(\lambda u) = \lambda g(\lambda\|u\|_{sc}) u$ and hence $[0,Tu]\subset K$. Given $u,v\in C$ we set $\alpha := g(\|u\|_{sc})$, $\beta := g(\|v\|_{sc})$ and let $z\in [Tu,Tv]=[\alpha u, \beta v]$, in other words, there is $s\in [0,1]$ with
  \[
    z = s\alpha u + (1-s)\beta v.
  \]
  We pick a $t\in [0,1]$, set 
  \[
    y = tu +(1-t)v, \qquad \gamma = \frac{\alpha\beta}{t\beta+(1-t)\alpha} \qquad\text{and note that}\qquad z = \gamma y.
  \]
  Since
  \begin{align*}
    g(\|y\|_{sc}) &= g (\|tu+(1-t)v\|_{sc}) > g(t\|u\|_{sc}+(1-t)\|v\|_{sc}) \\ &= \frac{g(\|u\|_{sc})g(\|v\|_{sc})}{tg(\|v\|_{sc}) + (1-t) g(\|u\|_{sc})} = \frac{\alpha\beta}{t\beta+(1-t)\alpha} = \gamma
  \end{align*}
  we have $z = \gamma y \in [0, Ty) \subset K$. In particular $K$ is a convex set. Note that $0$ is a non-support point of $C$ and hence for every hyperplane $H$ containing $0$ there are $y_1,y_2\in C$ which are on different sides of $H$. By construction, $K$ contains a positive multiple of both of these points and hence $0$ is a non-support point of~$K$. Since a point which is strictly between two points in~$K$ can only be a support point if the endpoints of this segment are support points and $0$ is a non-support point, we conclude that $z$ is a non-support point.

  In other words, we have shown that no non-trivial convex combination of points in $K$ can be a support point. Since for every supporting hyperplane $H$ the set $H\cap K$ is a convex set of support points of $K$ it has to be a singleton, i.e. $K$ is rotund.

  Since for $w\in C$ we have
  \[
    g(\|w\|_{sc}) = \frac{1}{1+a\|w\|_{sc}} \geq \frac{1}{1+\frac{\varepsilon}{2(1-\varepsilon)}} = \frac{1-\varepsilon}{1-\frac{\varepsilon}{2}}
  \]
  we obtain
  \[
    (1-\varepsilon) w = \frac{1-\varepsilon}{1-\frac{\varepsilon}{2}} \left(1-\frac{\varepsilon}{2}\right)w \in \left[0,T\left(\left(1-\frac\varepsilon2\right)w\right)\right] \subset K
  \]
  and hence $(1-\varepsilon) C \subset K \subset C$ as claimed.
\end{proof}

\begin{proposition}\label{prop:ConvexStrConv}
  Let $X$ be a separable Banach space and $C\subset X$ a bounded closed and convex set. For every $\varepsilon \in (0,1)$ there is a rotund set $K$ with $h(K,C)<\varepsilon$. If $C$ is compact, so is $K$.
\end{proposition}

The following proof is based on a modification of an argument by V.~Klee used in the proof of Theorem~1.5 in~\cite[p.~56]{Klee1959SmoothnessConvexity}.

\begin{proof}
  We pick a dense sequence $(w_k)_{k\in\mathbb{N}}$ in the unit sphere of $X$ and consider the linear mapping
  \[
    T\colon \ell_2 \to X, \qquad (a_k)_{k\in\mathbb{N}} \mapsto \sum_{k=1}^{\infty} 2^{-k} a_k w_k
  \]
  whose restriction to the closed unit ball $B_{\ell_2}$ is weak-to-norm continuous. Hence, the set $\frac{\varepsilon}{2}T(B_{\ell_2})$ is a compact subset of $X$. We set $\tilde{K}:= C + \frac{\varepsilon}{2}T(B_{\ell_2})$ and observe that as the sum of two closed convex sets, one of them compact, it is a closed convex set and since $0\in T(B_{\ell_2})$, it contains $C$. Note that if $C$ is compact, so is $\tilde{K}$. By construction we also have $h(\tilde{K},C)<\frac{\varepsilon}{2}$. Since $0$ is a non-support point of $T(B_{\ell_2})$, every point of $C$ is a non-support point of $\tilde{K}$. Now the claim follows from Proposition~\ref{prop:KleeStrConv}.
\end{proof}

\begin{corollary}\label{cor:TypicalConvSetSepSpace}
  Let $X$ be a separable Banach space and $C\subset X$ a closed and convex subset. The set of rotund compact sets is a dense subset of the hyperspace of compact convex sets.
\end{corollary}

\begin{remark}
  In~\cite{Klee1959SmoothnessConvexity}, V.~Klee shows that the set of convex compact sets which are $X^*$-rotund and $X^*$-smooth is the complement of a meagre set. Let us mention a view differences between these results.
  \begin{enumerate}
  \item Note that our result slightly differs from Klee's result, since our definition of rotundity is more restrictive than Klee's since we do require the intersection of every supporting hyperplane with the convex set to be a singleton and do not allow the case where this intersection is the whole set. In the latter case uniqueness of the metric projection is no longer guaranteed.
  \item The difference between our definition of rotund sets and V.~Klee's definition of $X^*$-rotund sets also explains why we need $X$ to be separable and Klee does not:
    \begin{enumerate}
    \item Every compact set is separable and so is its closed linear span. For this reason if $X$ is non-separable, every compact set $C$ is contained in a hyperplane and hence also in one of its supporting hyperplanes. This implies that the only rotund compact subsets of a non-separable Banach space are the singletons. For this reason Corollary~\ref{cor:TypicalConvSetSepSpace} fails in the non-separable space. Since, in the language of this paper, $X^*$-rotund sets may be subsets of supporting hyperplanes,
      the corresponding result by V.~Klee also holds in this case. 
    \item Every non-separable Banach space $X$ which is not rotund, contains a compact subset $K$ which is rotund in an infinite-dimensional subspace in which it is contained but fails to have unique metric projections: Since $X$ is not rotund, we may pick $x,y\in S_X$ with $[x,y]\subset S_X$. Using the Hahn-Banach Theorem we obtain a functional $x^*\in B_{X^*}$ with $\langle x^*, \frac{1}{2}(x+y)\rangle = 1$. A direct computation shows that this functional satisfies $\langle x^*,x\rangle = \langle x^*,y\rangle =1$. We now choose a separable infinite dimensional subspace $Y$ of $\ker x^*$ which contains $y-x$. Using the construction in the proof of Proposition~\ref{prop:ConvexStrConv}, we obtain an infinite dimensional compact convex set $\tilde{K}\subset Y$ containing $[0,y-x]$ which, as a subset of $Y$, is rotund. We finally set $K:=x+\tilde{K}$, $H:=\{z\in X^*\colon \langle x^*,z\rangle = 1\}$\linebreak and observe that $K\subset H$. Since $\|z\|\geq \langle x^*, z\rangle =1 $ for all $z\in H$ and hence in particular for all $z\in K$, we observe that both $x$ and $y$ are points in $K$ minimising the distance to the origin. Hence, $K$ fails to have unique projections.
    \end{enumerate}    
  \end{enumerate}
\end{remark}

Another direct consequence of Proposition~\ref{prop:ConvexStrConv} is the following useful variant where we require a fixed element of $C$ to be contained in the rotund set nearby.

\begin{corollary}\label{cor:RotundContainingPoint}
  Let $C$ be a convex compact set, $\varepsilon>0$, and $z\in C$. There is a rotund compact set $K$ with $z\in K$ and $h(C,K)<\varepsilon$.
\end{corollary}

\begin{lemma}\label{lem:diamPC}
  Let $X$ be a separable Banach space and $C\subset X$ a closed and convex set. For each $n\in\mathbb{N}$ the set $\mathcal{L}_n$ of compact convex subsets $K\subset C$ with the property, that there are points $x_n,y_n\in K$ with $\|x_n-y_n\|\geq 1/n$ and a norm one functional attaining its maximum both in $x_n$ and $y_n$, is closed and nowhere dense.
\end{lemma}

\begin{proof}
  This can be shown by repeating the proof of Theorem~2.2 in~\cite{Klee1959SmoothnessConvexity}.
\end{proof}

\begin{remark}
  Note that given $K\not\in\mathcal{L}_n$ and $x\in X$ the diameter of the set
  \[
    \{z\in K\colon \|x-z\|=d(x,K)\}
  \]
  is at most $1/n$ since otherwise we may pick $y,z\in K$ with $d(x,K)=\|x-y\|=\|x-z\|$ and $\|y-z\| > 1/n$. By Proposition~\ref{prop:HyperplaneAndProjection} there is a functional $x^*\in B_{X^*}$ with \[
    \langle x^*, \frac{1}{2}(y+z)\rangle \geq \langle x^*, w\rangle
  \]
  for all $w\in C$. A direct computation shows that $\langle x^*,y\rangle = \langle x^*,z\rangle = \langle x^*, (y+z)/2\rangle$ which contradicts $K\not\in\mathcal{L}_n$.
\end{remark}

As a conclusion, combining Proposition \ref{prop:ConvexStrConv} and Lemma \ref{lem:diamPC}, we have proved the following theorem, which plays a key role in the proof of the main result as it allows in our setting to always pass on to a rotund set which consequently ensures the projections to be unique.

\begin{theorem}\label{thm:rotundResidual}
  Let $X$ be a separable Banach space and $C\subset X$ a closed and convex set. The set of rotund compact subsets of $C$ is a residual subset of the hyperspace $\mathcal{CK}(C)$. In particular, the set of convex compact sets having a continuous (single-valued) metric projection is a residual subset of the hyperspace $\mathcal{CK}(C)$.
\end{theorem}

In addition, we obtain as a direct consequence of Lemma~\ref{lem:diamPC} and Proposition~\ref{prop:contMetrProjSaP} the following corollary that we frequently use in the last section.

\begin{corollary}\label{cor:ContinuityAtRotundSets}
  Let $C\subset X$ be a rotund compact set and $x\in X$. For every $\varepsilon>0$ and every $n\in\mathbb{N}$ there is $\delta>0$ such that for every compact convex set $K$ with $h(K,C)<\delta$ and every $y\in X$ with $\|x-y\|<\delta$ we have
  \[
    h(\{z\in K\colon \|y-z\|=d(y,K)\}, \{P_C x\}) < \varepsilon
  \]
  and the diameter of the set $\{z\in K\colon \|y-z\|=d(y,K)\}$ is at most $1/n$.
\end{corollary}

\section{Perturbation of compact-convex-valued mappings}
The aim of this section is to show that we are able to change the value of a set-valued mapping at a given point to a set nearby without increasing the Lipschitz constant too much. For this aim we need the following lemma from~\cite[Lemma~3.1]{Dymond2021} or~\cite[Lemma~5.3]{BRT2022}, cf. also Lemma~4.2 in~\cite{Medjic2022}.

\begin{lemma}\label{lem:CleverLemma}
  Let $X$ be a Banach space, $C\subset X$ a convex set containing the origin and $0<r<R$. There is a mapping $\Phi\colon C\to C$ such that
  \begin{enumerate}
  \item $\Phi(x)=0$ for all $x\in C\cap \overline{B}(0,r)$.
  \item $\Phi(x)=x$ for all $x\in C \setminus B(0,R)$.
  \item $\lip\Phi \leq 1 + \frac{r}{R-r}$.
  \item $\|\Phi(x)-x\|\leq r$ for all $x\in C$.
  \end{enumerate}
\end{lemma}

Together with Proposition~\ref{prop:ConvexLip} this lemma allows us to replace the value of a set-valued mapping at some point to compact and convex set close to the original value while having enough control over the Lipschitz constant of the resulting map.

\begin{lemma}\label{lem:PerturbationLemma}
  Let $F\in\mathcal{M}$, $C\subset X$ a closed convex set with $\xi\in C$, $0<\rho<r<R$ and $K$ a compact convex set satisfying $h(K,F(\xi))<\rho$. There is a mapping $G\colon C\to C$ with the following properties:
  \begin{enumerate}
  \item $G(\xi) = K$.
  \item $G(x)=F(x)$ for $x\in C\setminus B(\xi,R)$.
  \item $\lip G \leq \max\{\frac{R}{R-r}\lip F, \lip F + \frac{\rho}{r}\}$.
  \item $h(F(x),G(x))\leq 2r$ for all $x\in C$.
  \end{enumerate}
\end{lemma}

\begin{proof}
We start by defining a Lipschitz mapping
  \[
    \lambda\colon C \to [0,1], \qquad x\mapsto \max\left\{1-\frac{\|x-\xi\|}{r},0\right\}
  \]
  and choosing the mapping $\Phi$ from Lemma~\ref{lem:CleverLemma} corresponding to $r$ and $R$. Using these two mappings we define $G\colon C\to C$ by
  \[
    G(x) := \lambda(x) K + (1-\lambda(x)) F(\xi+\Phi(x-\xi)).
  \]
  In the rest of the proof we show that $G$ satisfies the conditions \textit{1}.-\textit{4}.. Obviously \textit{1.} holds. Furthermore, observe that $G(x)=F(x)$ for $x\in C\setminus B(\xi,R)$ since in this case we have $\lambda(x)=0$ and $\xi+\Phi(x-\xi)=x$, hence \textit{2.} holds.
  
  For $x,y\in B(\xi,r)$ we conclude from Proposition~\ref{prop:ConvexLip} that
  \[
    h(G(x),G(y)) \leq (\lip F + h(K, F(\xi))\lip \lambda) \|x-y\| \leq (\lip F + \rho/r) \|x-y\|.
  \]
  For $x,y\in C\setminus B(\xi,r)$ we have
  \begin{align*}
    h(G(x), G(y)) &= h(F(\xi+\Phi(x-\xi)), F(\xi+\Phi(y-\xi))) \leq \lip F\lip \Phi \|x-y\|\\
    &\leq \frac{R}{R-r}\lip F \|x-y\|
  \end{align*}
  by Lemma~\ref{lem:CleverLemma}. Bearing in mind Lemma \ref{lem: Lipboundinandoutofball} we see that taking the maximum of these bounds implies the claimed bound on the Lipschitz constant of~$G$, i.e. we have shown \textit{3.}.
  
As a last step we prove \textit{4}.. By Lemma~\ref{lem:Enpoints} we obtain
  \[
    h(F(x),G(x)) \leq h(F(x),F(\xi)) + h(F(\xi), \lambda(x)K+(1-\lambda(x))F(\xi)) \leq r + \rho \leq 2r
  \]
  for all $x\in C\cap B(\xi,r)$. For $x\in C\setminus B(\xi,r)$ Lemma~\ref{lem:CleverLemma} implies that
  \[
    h(F(x),G(x)) = h(F(x), F(\xi+\Phi(x-\xi))) \leq r \lip F  < 2r
  \]
  which finishes the proof.
\end{proof}

\section{Successive Approximations for compact convex valued mappings}
The goal of the last section is to prove the main result of this paper, more precisely, to show that for the typical nonexpansive mapping with convex and compact point images and typical points of its domain the sequence of successive approximations is unique and converges to a fixed point of the mapping.

Let us first fix a few necessary notions. For a compact set $M$ we denote by
\[
  \tilde{P}_M x := \{z\in M\colon \|x-z\|=d(x,M)\}
\]
the set-valued metric projection onto $M$.

\begin{definition} 
	Let $X$ be a Banach space, $C\subset X$ a closed and bounded subset, \linebreak $F:C\to \mathcal{K}(C)$ and $N\in \mathbb{N}$. A \textit{trajectory with respect to} $F$ or \textit{a sequence of successive approximations with respect to} $F$ is a sequence $(x_i)_{i=1}^N $ in $C$ with
	\[ 
	x_{n+1}\in \tilde P_{F(x_n)} x_n, \quad n=1,\ldots, N-1.
	\]
\end{definition}
\begin{definition}
	An infinite sequence of successive approximations or trajectory $(x_n)_{n=0}^\infty$ is called \emph{regular}, if $P_{F(x_n)} x_n$ is a singleton for all $n$.
\end{definition}
\begin{theorem}\label{thm:MainResult1}
  Let $X$ be a separable Banach space, $C\subset X$ a bounded closed and convex set, and let $x_0\in C$. There is a set $\mathcal{A}(x_0)\subset\mathcal{M}$ with meagre complement with the property that for every $F\in\mathcal{A}(x_0)$ the sequence $x_{n+1} :=P_{F(x_n)} x_n$ is well defined, i.e. unique, and converges to a fixed point of~$F$.
\end{theorem}

We start with preliminary work, and combine the preliminary steps into the proof of the theorem later.
Let $C$ be a bounded closed and convex subset of a separable Banach space $X$ and $x_0\in C$. (We fix $X$, $C$ and $x_0$ for the whole section.) 
Define
\begin{multline*}
  \mathcal{A}_n(x_0) := \Big\{F\in\mathcal{M}\colon \exists r >0 \;\text{s.t.}\; \diam \tilde{P}_{F(v_k)}v_k \leq \frac{1}{n} \; \text{for all}\;v_0,\ldots,v_n\;\text{with}\;\\ v_0\in B(x_0,r), v_{k+1}\in\tilde{P}_{F(v_k)} v_k\Big\}.
\end{multline*}
Our initial aim is to show that each $\mathcal{M}\setminus \mathcal{A}_n(x_0)$ is a nowhere dense set. In order to do so, we follow the upcoming scheme:

\begin{itemize}
\item[I] First induction (moving forward) to obtain the sequence $(y_n)\subset C$.
\item[II] Second induction (moving backwards) to obtain sequences $(z_n)\subset C$ and $(G_n)\subset \mathcal{M}.$
\item[III]  Proof that $\mathcal{M}\setminus \mathcal{A}_n(x_0)$ is nowhere dense.
\end{itemize}

\subsection*{I Induction: moving forward}
\subsubsection*{I Induction base}
Let $F\colon C\to \mathcal{CK}(C)$ be a nonexpansive mapping and $\varepsilon>0$. Since by Proposition \ref{prop: NdenseinM} the set of strict contractions is a dense subset of $\mathcal{M}$, we may assume without loss of generality that $\lip F < 1$. More precisely, we may choose a strict contraction whose distance to $F$ is smaller than $\varepsilon/2$ and then proceed with $\varepsilon/2$ instead of $\varepsilon$.

We now define a mapping $\tilde{F}$ which is close to $F$ and a sequence $(y_n)_{n\in\mathbb{N}}$ in $C$ with $y_{n+1}\in \tilde{P}_{\tilde{F}(y_n)} y_n$.
We start by setting $y_0 := x_0$ and need to do an additional step. The importance of this step will become clear at the end of the construction and serves the purpose of anchoring the sequence of $(z_k)_{k=1}^{n}$ which will be constructed later. If $F(x_0)$ is rotund, we set $\tilde{F}=F$. Otherwise, we deduce from Theorem~\ref{thm:rotundResidual} that there is a compact rotund set $K_0$ with $h(F(x_0), K_0)<\rho_0$ where $\rho_0\in (0, \frac{\varepsilon}{4}(1-\lip F))$. If $x_0\in F(x_0)$, by Corollary~\ref{cor:RotundContainingPoint} we may assume that $x_0\in K_0$.
These considerations allow us to choose a large enough $R_0>0$ to apply Lemma~\ref{lem:PerturbationLemma} and obtain a nonexpansive mapping $\tilde{F}$ with $\tilde{F}(x_0)=K_0$, $\lip\tilde{F}<1$ and $d_\infty(F,\tilde{F})<\frac{\varepsilon}{2}$. We set $y_1 := P_{\tilde{F}(y_0)} y_0$ and proceed inductively.

\subsubsection*{I Induction step}
Assume that all elements up to $y_n$ have been defined already. If $y_n\in \tilde{F}(y_n)$, we set $y_{n+1}=y_n$. Otherwise we observe that again by Theorem~\ref{thm:rotundResidual} there is a sequence of compact rotund sets $(K_k^n)_{k\in\mathbb{N}}$ with $K_k^n\to \tilde{F}(y_n)$ for $k\to\infty$ with respect to the Hausdorff distance. Since the sequence $P_{K_k^n} y_n$ is contained in a compact set we may, by passing to subsequences if necessary, assume without loss of generality, that the limit
\[
  y_{n+1} := \lim_{k\to\infty} P_{K_k^n} y_n
\]
exists.

\subsection*{II Induction: moving backwards}
\subsubsection*{II Induction base}
Using this sequence, for fixed $n\in\mathbb{N}$ we construct a mapping $G\in \mathcal{A}_n(x_0)$ with distance to $\tilde{F}$ at most~$\varepsilon$. Observe that since
\[
  \|y_{k+1}-y_k\|=d(y_k,\tilde{F}(y_k))\leq h(\tilde{F}(y_{k-1}),\tilde{F}(y_k)) \leq \lip \tilde{F} \|y_k-y_{k-1}\|
\]
and $\lip \tilde{F} < 1$, all $y_k$ are different as long as none of them is a fixed point, see also Lemma~4.1 in~\cite{Medjic2022}. If we hit a fixed point, the sequence remains constant after that. Hence 
\[
  c := \min\{\|y_i-y_j\|\colon 0\leq i,j\leq k, k=\max\{m\colon y_m\neq y_n\}\}>0.
\]
We pick $0 < \rho_n < r_n < R_n < \frac{c}{2}$ with
\begin{equation}
  r_n< \frac{\varepsilon}{8n}, \qquad \lip \tilde{F} \frac{R_n}{R_n-r_n} < 1 \qquad \text{and}\qquad \lip \tilde{F} + \frac{\rho_n}{r_n} < 1.
\end{equation}
We choose $k$ large enough so that $h(K_k^{n},\tilde{F}(y_{n})) < \rho_n$ and denote by $G_n$ the mapping we obtain from Lemma~\ref{lem:PerturbationLemma} using the above parameters. Note that the above assumptions imply that $G_n$ is a strict contraction and $G_n(y_n)$ is rotund. Moreover, we did not change the values on previous points of the sequence, i.e. $G_n(y_k)=\tilde{F}(y_k)$ for all $0\leq k < n$.

Now Corollary~\ref{cor:ContinuityAtRotundSets} allows us to choose $\delta_n>0$ such that $\diam\tilde{P}_{K} x \leq \frac{1}{n}$ for every convex compact set~$K$ with $h(K,G_n(y_n))<\delta_n$ and every $x\in C\cap B(y_n,\delta_n)$. 

We set $z_n:=y_n$, pick an arbitrary $z_{n+1}\in G_n(z_n)$ and set $\delta_{n+1}:=\frac{1}{n}$. The only role of the parameters $z_{n+1}$ and $\delta_{n+1}$ is to facilitate the phrasing of Lemma~\ref{lem:IndStep} and they do not have a meaning beyond that. Now we are able to proceed inductively.

\subsubsection*{II Induction step}
\begin{lemma}\label{lem:IndStep}
  Let $1 < k \leq n$, $z_{k},\ldots,z_n\in C$, $G_k\in \mathcal{M}$ and $\delta_n> \ldots > \delta_k>0$ with the following properties be given.
  \begin{enumerate}[(i)]
  \item $\lip G_k < 1$.
  \item $d_\infty(\tilde{F},G_k)<\frac{\varepsilon(n+1-k)}{4n}$.
  \item $\|y_k-z_k\|< \delta_k/4$.
  \item For $k\leq m\leq n$ and for every compact convex set~$K$ with $h(K,G_k(z_m))<\delta_m$ and every $x\in C\cap B(z_m,\delta_m)$, we have $\diam \tilde{P}_{K} x \leq \frac{1}{n}$ and $\tilde{P}_{K} x \subset C\cap B(z_{m+1},\frac{\delta_{m+1}}{2})$.
  \item $G_k(x)=\tilde{F}(x)$ for all $x\in C\setminus \bigcup_{m=k}^{n} B(y_m,c/2)$.
  \end{enumerate}
  Then there is a mapping $G_{k-1}$, a point $z_{k-1}$ and $\delta_{k-1}>0$ such that
  \begin{enumerate}[(1)]
  \item $\lip G_{k-1} < 1$.
  \item $d_\infty(\tilde{F},G_{k-1})<\frac{\varepsilon(n+1-(k-1))}{4n}$.
  \item $\|y_{k-1}-z_{k-1}\|< \delta_{k-1}/4$.    
  \item For $k-1\leq m\leq n$ every compact convex set~$K$ with $h(K,G_{k-1}(z_m))<\delta_{m}$ and every $x\in C\cap B(z_m,\delta_{m})$ satisfy $\diam \tilde{P}_{K} x \leq \frac{1}{n}$ and $\tilde{P}_{K} x \subset B(z_{m+1},\frac{\delta_{m+1}}{2})$.
  \item $G_{k-1}(x)=\tilde{F}(x)$ for all $x\in C\setminus \bigcup_{m=k-1}^{n} B(y_m,c/2)$.    
  \end{enumerate}  
\end{lemma}

\begin{proof}
  Let $1< k \leq n$, $z_{k},\ldots,z_n\in C$, $G_k\in \mathcal{M}$ and $\delta_n> \ldots > \delta_k>0$ satisfy the conditions \textit{(i)}-\textit{(v)}. We show that \textit{(1)}-\textit{(5)} hold. 
  For that aim pick
  \[
    0 < \rho_{k-1} < r_{k-1} < R_{k-1} < \frac{c}{2}
  \]
  with
  \[
    r_{k-1}< \frac{\varepsilon}{8n}, \qquad \lip G_{k} \frac{R_{k-1}}{R_{k-1}-r_{k-1}} < 1 \qquad \text{and}\qquad \lip  G_{k} + \frac{\rho_{k-1}}{r_{k-1}} < 1.
  \]  
  By definition of the sequence $(y_m)_{m\in\mathbb{N}}$ there is a $j\in\mathbb{N}$ with
  \[
    \|y_{k}-P_{K_j^{k-1}} y_{k-1}\| < \frac{\delta_k}{8}\quad\text{and}\quad h\left(\tilde{F}(y_{k-1}),K_j^{k-1}\right)<\frac{\rho_{k-1}}{2}
  \]
  By Corollary~\ref{cor:ContinuityAtRotundSets} there is $s \in (0,\min\{\rho_{k-1},\delta_k\})$ such that
 \begin{equation}\tag{a}
    \diam \tilde{P}_{K} y \leq \frac{1}{n} \qquad\text{and}\qquad h(\tilde{P}_K y,\{P_{K_j^{k-1}} y_{k-1}\})<\delta_k/8
 \end{equation} 
  for all $y\in B(y_{k-1}, s)$ and all compact convex sets $K$ with $h(K,K_{j}^{k-1})<s$.
  We pick a point $z_{k-1} \in B(y_{k-1},s/8)$, set $\delta_{k-1}:=s/2$ and note that $\|y_{k-1}-z_{k-1}\| < \delta_{k-1}/4$ which means that \textit{(3)} is satisfied and in particular implies
  \[
    \|y_m-z_{k-1}\| \geq \|y_m-y_{k-1}\|-\|y_{k-1}-z_{k-1}\| \geq c - \frac{\rho_{k-1}}{2} \geq \frac{c}{2}
  \]
  for $m\geq k$ and hence $G_{k}(z_{k-1})=\tilde{F}(z_{k-1})$. Moreover, we have  
\begin{equation}\tag{b}
h(\tilde{F}(z_{k-1}),K_j^{k-1})\leq h(\tilde{F}(z_{k-1}), \tilde{F}(y_{k-1})) + h(\tilde{F}(y_{k-1}), K_j^{k-1})< \frac{s}{8} + \frac{\rho_{k-1}}{2} < \rho_{k-1}
\end{equation}
  and
  \begin{align*}\tag{c}
    h(\tilde{P}_K y,\{z_k\}) &\leq h(\tilde{P}_K y,\{P_{K_j^{k-1}} y_{k-1}\}) +\|P_{K_j^{k-1}} y_{k-1}-y_k\| + \|y_k-z_k\| \\ & <  \frac{\delta_k}{8} + \frac{\delta_k}{8} + \frac{\delta_k}{4} = \frac{\delta_k}{2}.
  \end{align*}
  Due to our choice of the parameters $\rho_{k-1},r_{k-1}, R_{k-1}, z_{k-1},$ and $K_j^{k-1}$ and inequality~$(b)$, we obtain a strict contraction $G_{k-1}$ from Lemma~\ref{lem:PerturbationLemma} which satisfies $G_{k-1}(z_{k-1})=K_{j}^{k-1}$, $G_{k-1}(x)=G_k(x)$ for all $x\in C\setminus B(z_{k-1},R_{k-1})$, $\lip G_{k-1}<1$, and
  \[
    d_\infty(G_{k-1},\tilde{F}) \leq d_\infty(G_{k-1},G_{k})+d_\infty(G_k,\tilde{F}) \leq \frac{\varepsilon}{4n} + \frac{\varepsilon(n+1-k)}{4n} = \frac{\varepsilon(n+1-(k-1))}{4n}.
  \]
  Note that conditions \textit{(1)} and \textit{(2)} are now also satisfied. To see \textit{(5)}, note that since $G_k(x)=\tilde{F}(x)$ for an even larger set by \textit{(v)}, we also have this identity for all \linebreak $x\in C\setminus \bigcup_{m=k-1}^{n} B(y_m,c/2).$
  
  To end the proof, we confirm \textit{(4)}. Given a compact convex set $K$ with \linebreak $h(K,G_{k-1}(z_{k-1}))<\delta_{k-1}$ and a point $x\in C \cap B(z_{k-1}, \delta_{k-1})$ we have \linebreak $\|x-y_{k-1}\| < \|x-z_{k-1}\| + \frac{\delta_{k-1}}{4} \leq 2 \delta_{k-1} = s$ and hence, as $G_{k-1}(z_{k-1})=K_{j}^{k-1}$ the choice of $s$ above ensures $(a)$ and $(c),$ therefore~\textit{(4)} is satisfied for $m=k-1$. Since $G_{k-1}(z_{m})=G_{k}(z_{m})$, \textit{(4)} is also satisfied for $m\geq k$ since in that range \textit{(iv)} holds. This completes the proof.
\end{proof}

\subsubsection*{III $\mathcal{M}\setminus \mathcal{A}_n(x_0)$ is nowhere dense}
Relying on the previous construction established in I and II, we now show the following.

\begin{proposition}\label{prop:closemapping}
  Given $F\in\mathcal{M}$, $x_0\in C,$ and $\varepsilon>0$ there is a strict contraction $G\in \mathcal{M}$, $r>0$ and $\delta>0$ with the following properties.
  \begin{enumerate}
  \item $d_\infty(F,G)<\varepsilon$
  \item $\diam \tilde{P}_{H(v_k)} v_k\leq \frac{1}{n}$ for all $v_0,\ldots,v_n$ with $v_0\in B(x_0,r)$, $v_{k+1}\in\tilde{P}_{H(v_k)} v_k$ and every $H\in\mathcal{M}$ with $h(H,G)<\delta$.
  \end{enumerate}
  In particular, $B_{\mathcal{M}}(G,\delta) \subset \mathcal{A}_{n}(x_0) \cap B_\mathcal{M}(F,\varepsilon)$, i.e. $\mathcal{M}\setminus \mathcal{A}_n(x_0)$ is nowhere dense.
\end{proposition}

\begin{proof}
  Using the construction established with I and II, we obtain a mapping $G_1$ with
  \[
    d_\infty(F,G_1) \leq d_\infty(F, \tilde{F}) + d_\infty(\tilde{F}, G_1) < \frac{\varepsilon}{2} + \frac{\varepsilon}{4} = \frac{3\varepsilon}{4} < \varepsilon,
  \]
  numbers $0<\delta_1<\ldots<\delta_n$ and a sequence of points $z_1,\ldots, z_n$ with $\|y_k-z_k\|< \frac{\delta_k}{4}$ and for every $1\leq m\leq n$ and every compact convex set~$K$ with $h(K,G_{1}(z_m))<\delta_{m}$ and every $x\in C\cap B(z_m,\delta_{m})$ we have $\diam \tilde{P}_{K} x \leq \frac{1}{n}$ and $\tilde{P}_{K} x \subset B(z_{m+1},\delta_{m+1}/2)$. We set $z_0:=x_0$ and $G:=G_1$.
  Note that $G$ satisfies condition \textit{1.}. The rest of the proof is dedicated to showing that \textit{2.} holds.

  By Corollary~\ref{cor:ContinuityAtRotundSets} there is a $\delta_0\in (0,\delta_1/2)$ such that $\tilde{P}_K x \subset C\cap B(z_1,\delta_1/2)$ for every $x\in B(z_0,\delta_0)$ and every compact convex set $K$ with $h(K,G(z_1))<\delta_0$. We set $r:=\delta_0/2$ and $\delta:=\min\{r,\frac{\varepsilon}{4}\}$.
  
  Let $H\in B_\mathcal{M}(G,\delta)$ and $v_0\in B(z_0,r)$. Since $H$ is nonexpansive, we have
  \[
    h(H(v_0),G(z_0)) \leq h(H(z_0),G(z_0)) + \|v_0-z_0\| < \delta_0
  \]
  which by the choice of these constants implies that $\tilde{P}_{H(v_0)} v_0 \subset C\cap B(z_1,\delta_1/2)$ and $\diam\tilde{P}_{H(v_0)} v_0\leq 1/n$. We continue by induction and use the properties established in Lemma~\ref{lem:IndStep}. Assume we have
  \[
    v_k\in \tilde{P}_{H(v_{k-1})} v_{k-1} \subset C\cap B(z_k,\delta_k/2)
  \]
  then we get
  \[
    h(H(v_k),G(z_k)) \leq h(H(z_k),G(z_k)) + \|v_k-z_k\| \leq \delta_k
  \]
  and hence $\diam\tilde{P}_{H(v_k)} v_k\leq 1/n$ and $\tilde{P}_{H(v_{k+1})} v_{k+1} \subset C\cap B(z_{k+1},\delta_{k+1}/2)$. As long as $k+1\leq n$. Even though, in the last step, we only obtain that $\diam\tilde{P}_{H(v_n)} v_n \leq 1/n$ but no additional information on the location of this set, it is enough to show that $B_{\mathcal{M}}(G,\delta) \subset \mathcal{A}_{n}(x_0) \cap B_\mathcal{M}(F,\varepsilon)$ and hence that $\mathcal{M}\setminus \mathcal{A}_n(x_0)$ is nowhere dense.
\end{proof}

We have now completed the preliminary work, and can move on to proving the main theorem.
\begin{proof}[Proof of Theorem~\ref{thm:MainResult1}]
  We set
  \[
    \mathcal{A}(x_0) = \bigcap_{n=1}^{\infty} \mathcal{A}_n(x_0)
  \]
  and observe that it is the complement of a meagre set. Moreover for $F\in \mathcal{A}(x_0)$ and a sequence $(x_k)_{k\in\mathbb{N}}$ starting with $x_0$ and satisfying $x_{k+1}\in \tilde{P}_{F(x_k)} x_k$ we have $\diam \tilde{P}_{F(x_k)} x_k \leq 1/n$ for every $n\in\mathbb{N}$ with $n\geq k$ and hence $\tilde{P}_{F(x_k)} x_k$ is a singleton. Convergence to a fixed point follows from Theorems 4.2, 4.3 and 4.4 in~\cite{RZ2002}.		
\end{proof}

Lastly, we formalise the following consequence of the main theorem and its proof.

\begin{corollary}
  Let $X$ be a separable Banach space. Then there is a residual set $\mathcal{F}\subset \mathcal{M}$ such that for every $F\in \mathcal{F}$ there is a residual subset $\mathcal{U}\subset C$ with the property that for every $x_0\in \mathcal{U}$ the mapping $F$ admits a regular trajectory starting at $x_0$.
\end{corollary}
\begin{proof}
  Let $D$ be a countable and dense subset of $C$. For every $n\in \mathbb N$ define the sets
  \[
    \mathcal{A}_n := \bigcap_{y\in D} \mathcal A_n(y) \quad \text{ and }\quad \mathcal{F}:= \bigcap_{n=0}^\infty \mathcal A_n.
  \]
  Observe that since $D$ is countable the set $\mathcal A_n$ is residual and hence also the set $\mathcal{F}$ is residual in $\mathcal{M}$. For $F\in \mathcal{F}$ and $n\in \mathbb N$ we have in particular $F\in \mathcal A_n(y)$ for every $y\in D$. Hence there is an $r_{n,y}>0$ such that $\diam \tilde{P}_{F(v_k)}v_k \leq \frac{1}{n}$ for all $v_0,\ldots,v_n$ with $v_0\in B(y,r_{n,y})$ and $v_{k+1}\in\tilde{P}_{F(v_k)} v_k$. Then for every $n\in \mathbb N$ the sets
  \[
    \mathcal{U}_n:= \left\{ x\in C \colon \|x-y\| < r_{n,y} \text{ for all }y\in D  \right\}
  \]
  are dense in $C$, since $D$ itself is dense in $C$. Hence
  \[
    \mathcal{U} := \bigcap_{n=1}^\infty \mathcal U_n
  \]
  is residual in $C$ and every $x_0\in \mathcal U$ has a regular trajectory with respect to $F$.
\end{proof}
\medskip\medskip{}

\noindent\textbf{Acknowledgement.} This research has been supported by the Austrian Science Fund (FWF): P~32523-N. This version of the article has been accepted for publication, after peer review 
but is not the Version of Record and does not reflect post-acceptance improvements, or any
corrections. The Version of Record is available online at:  \burl{https://doi.org/10.1007/s00605-022-01813-y}.

\vspace{6mm}
\noindent
Christian Bargetz\\
Department of Mathematics\\
University of Innsbruck\\
Technikerstraße 13,
6020 Innsbruck,
Austria\\
\texttt{christian.bargetz@uibk.ac.at}\\[3mm]
\noindent
Emir Medjic\\
Department of Mathematics\\
University of Innsbruck\\
Technikerstraße 13,
6020 Innsbruck,
Austria\\
\texttt{emir.medjic@uibk.ac.at}\\[3mm]
\noindent
Katriin Pirk\\
Department of Mathematics\\
University of Innsbruck\\
Technikerstraße 13,
6020 Innsbruck,
Austria\\
\texttt{katriin.pirk@uibk.ac.at}\\[3mm]


\begin{thebibliography}{10}

\bibitem{BMT2021:MetriProjections}
Vitor Balestro, Horst Martini, and Ralph Teixeira.
\newblock Convex analysis in normed spaces and metric projections onto convex
  bodies.
\newblock {\em J. Convex Anal.}, 28(4):1223--1248, 2021.

\bibitem{BR2016}
Christian Bargetz and Michael Dymond.
\newblock {$\sigma$}-porosity of the set of strict contractions in a space of
  non-expansive mappings.
\newblock {\em Israel J. Math.}, 214(1):235--244, 2016.

\bibitem{BDR2017}
Christian Bargetz, Michael Dymond, and Simeon Reich.
\newblock Porosity results for sets of strict contractions on geodesic metric
  spaces.
\newblock {\em Topol. Methods Nonlinear Anal.}, 50(1):89--124, 2017.

\bibitem{BR2019}
Christian Bargetz and Simeon Reich.
\newblock Generic convergence of sequences of successive approximations in
  {B}anach spaces.
\newblock {\em Pure Appl. Funct. Anal.}, 4(3):477--493, 2019.

\bibitem{BRT2022}
Christian Bargetz, Simeon Reich, and Daylen Thimm.
\newblock Generic properties of nonexpansive mappings on unbounded domains.
\newblock Preprint (arXiv:2204.10279), 2022.

\bibitem{Bou1998:Topology1-4}
Nicolas Bourbaki.
\newblock {\em General topology. {C}hapters 1--4}.
\newblock Elements of Mathematics (Berlin). Springer-Verlag, Berlin, 1998.

\bibitem{BM1976}
Francesco~S. De~Blasi and J\'{o}zef Myjak.
\newblock Sur la convergence des approximations successives pour les
  contractions non lin\'{e}aires dans un espace de {B}anach.
\newblock {\em C. R. Acad. Sci. Paris S\'{e}r. A-B}, 283(4):Aiii, A185--A187,
  1976.

\bibitem{BMRZ2009}
Francesco~S. de~Blasi, J\'{o}zef Myjak, Simeon Reich, and Alexander~J.
  Zaslavski.
\newblock Generic existence and approximation of fixed points for nonexpansive
  set-valued maps.
\newblock {\em Set-Valued Var. Anal.}, 17(1):97--112, 2009.

\bibitem{Dymond2021}
Michael Dymond.
\newblock Porosity phenomena of non-expansive, banach space mappings.
\newblock Preprint (arXiv:2110.13722), to appear in Israel Journal of
  Mathematics, 2021.

\bibitem{Gruber}
Peter~M. Gruber.
\newblock Results of {B}aire category type in convexity.
\newblock In {\em Discrete geometry and convexity ({N}ew {Y}ork, 1982)}, volume
  440 of {\em Ann. New York Acad. Sci.}, pages 163--169. New York Acad. Sci.,
  New York, 1985.

\bibitem{Klee1959SmoothnessConvexity}
Victor Klee.
\newblock Some new results on smoothness and rotundity in normed linear spaces.
\newblock {\em Math. Ann.}, 139:51--63 (1959), 1959.

\bibitem{Klee1953Homeomorphisms}
Victor~L. Klee, Jr.
\newblock Convex bodies and periodic homeomorphisms in {H}ilbert space.
\newblock {\em Trans. Amer. Math. Soc.}, 74:10--43, 1953.

\bibitem{Kohlenbach}
Ulrich Kohlenbach.
\newblock Some logical metatheorems with applications in functional analysis.
\newblock {\em Trans. Amer. Math. Soc.}, 357(1):89--128, 2005.

\bibitem{Medjic2022}
Emir Medjic.
\newblock On successive approximations for compact-valued nonexpansive
  mappings.
\newblock Preprint (arXiv:2203.03470), 2022.

\bibitem{PL2015}
Li-Hui Peng and Xian-Fa Luo.
\newblock Contractive set-valued maps in hyperbolic spaces.
\newblock {\em J. Nonlinear Convex Anal.}, 16(7):1415--1424, 2015.

\bibitem{Pianigiani}
G.~Pianigiani.
\newblock Generic properties of successive approximations in {H}ilbert spaces.
\newblock {\em J. Nonlinear Convex Anal.}, 16(6):1097--1111, 2015.

\bibitem{Rakotch}
E.~Rakotch.
\newblock A note on contractive mappings.
\newblock {\em Proc. Amer. Math. Soc.}, 13:459--465, 1962.

\bibitem{RZ2001}
Simeon Reich and Alexander~J. Zaslavski.
\newblock The set of noncontractive mappings is {$\sigma$}-porous in the space
  of all nonexpansive mappings.
\newblock {\em C. R. Acad. Sci. Paris S\'{e}r. I Math.}, 333(6):539--544, 2001.

\bibitem{RZ2002}
Simeon Reich and Alexander~J. Zaslavski.
\newblock Convergence of iterates of nonexpansive set-valued mappings.
\newblock In {\em Set valued mappings with applications in nonlinear analysis},
  volume~4 of {\em Ser. Math. Anal. Appl.}, pages 411--420. Taylor \& Francis,
  London, 2002.

\bibitem{Strobin2014}
Filip Strobin.
\newblock {$\sigma$}-porous sets of generalized nonexpansive mappings.
\newblock {\em Fixed Point Theory}, 15(1):217--232, 2014.

\bibitem{Wulbert1968}
D.~E. Wulbert.
\newblock Continuity of metric projections.
\newblock {\em Trans. Amer. Math. Soc.}, 134:335--341, 1968.

\bibitem{Zamfirescu}
Tudor Zamfirescu.
\newblock Nearly all convex bodies are smooth and strictly convex.
\newblock {\em Monatsh. Math.}, 103(1):57--62, 1987.

\end{thebibliography}
\end{document}